\documentclass[reqno]{amsart}
\usepackage[utf8]{inputenc}
\usepackage[T1]{fontenc}
\usepackage{lmodern}
\usepackage[english]{babel}
\usepackage{amsmath,a4wide}
\usepackage{xfrac}
\usepackage{esint}
\usepackage{mathrsfs,bm,amsthm,mathtools,yfonts,amssymb,color,braket,booktabs,graphicx,graphics,amsfonts,comment,siunitx}

\newcommand{\R}{\mathbb{R}}

\newcommand{\N}{\mathbb{N}}
\newcommand{\V}{\mathbf{V}}

\newcommand{\IV}{\mathbf{IV}}
\newcommand{\e}{\mathbf{e}}
\newcommand{\lv}{\lvert}
\newcommand{\rv}{\rvert}
\newcommand{\lV}{\lVert}
\newcommand{\rV}{\rVert}

\DeclareMathOperator{\spt}{spt}

\DeclareMathOperator{\dv}{div}

\DeclareMathOperator{\ess}{ess\,sup}
\DeclareMathOperator{\Tan}{Tan}

\theoremstyle{plain}
\newtheorem{theorem}{Theorem}[section]
\newtheorem*{theorem*}{Theorem}
\newtheorem{lemma}[theorem]{Lemma}
\newtheorem*{lemma*}{Lemma}
\newtheorem{prop}[theorem]{Proposition}
\newtheorem*{prop*}{Proposition}

\newtheorem*{corollary*}{Corollary}
\theoremstyle{definition}

\newtheorem*{example*}{e.g.}
\newtheorem{remark}[theorem]{Remark}
\newtheorem*{remark*}{Remark}
\newtheorem{assumption}[theorem]{Assumption}
\newtheorem*{assumption*}{Assumption}
\numberwithin{equation}{section}
\newtheorem{definition}[theorem]{Definition}
\newtheorem*{definition*}{Definition}
	
\title[Generalized BV mean curvature flow with a critical forcing term]{A note on generalized BV mean curvature flow\\
	with a critical forcing term}
\author[K. Tashiro]{Kiichi Tashiro}
\address{Dipartimento di Matematica, Universit\`{a} degli Studi di Milano, Via Saldini 50, I-20133 Milano (MI), Italy}
\email{kiichi.tashiro@unimi.it}

\begin{document}

\begin{abstract}
	In this note, we show that a weak mean curvature flow with critical forcing term obtained by Liu--Tonegawa (2024) satisfies the BV-type area change formula, that is, their flow is not only a Brakke flow but also a generalized BV flow, which is proposed by Stuvard--Tonegawa (2024). To establish this result, we identify minimal conditions under which a Brakke flow satisfies the area-change formula. As an application of our main theorem, we derive a lower bound for the extinction time of generalized BV flows with a critical forcing term. We also outline the proof of a compactness theorem for generalized BV flows.
\end{abstract}

\maketitle

\section{Introduction}
\label{sec: Intro}
Given an $ n $-dimensional submanifold $ \Gamma_0 \subset \R^{ n + 1 } $ and a vector field $ u : \R^{ n + 1 } \times [ 0 , \infty ) \to \R^{ n + 1 } $, the mean curvature flow (henceforth referred to MCF) with the vector field $ u $ is a one-parameter family $ \{ \Gamma ( t ) \}_{ t \geq 0 } $ of $ n $-dimensional manifolds in the Euclidean space $ \R^{ n + 1 } $ such that $ \Gamma ( 0 ) = \Gamma_0 $ and $ \Gamma ( t ) $ moves by the following motion law:
\[
	v ( x , t ) = h ( x , t ) + u^{ \perp } ( x , t ) \quad \text{ at } x \in \Gamma ( t ) , \ t > 0,
\]
where $ v ( x , t ) $ denotes the normal velocity vector of $ \Gamma ( t ) $ at $ x $, $ h ( x , t ) $ denotes the mean curvature of $ \Gamma ( t ) $ at $ x $, and $ u^{ \perp } ( x , t ) $ is the projection of $ u ( x , t ) $ onto $ ( \Tan_x \Gamma ( t ) )^{ \perp } $. When the given vector field $ u $ is not smooth, it is interesting to investigate  the best properties that can be expected, as well as the minimal requirements for the existence of solutions to the MCF.

In the case of $ n = 1 $, Liu--Tonegawa \cite{liu2024existence} recently established a time-global existence theorem of MCFs in Brakke's sense (see \cite{brakke1978motion}) with $ u $ satisfying the following ``critical'' condition:
\begin{equation}
\label{eq: critical}
	u \in L_{ loc }^{ \infty } ( [ 0 , \infty ) ; L^2 ( \R^2 ; \R^2 ) ) \cap L^2_{ loc } ( [ 0 , \infty ) ; W^{ 1 , 2 } ( \R^2 ; \R^2 ) ).
\end{equation}
The condition of criticality \eqref{eq: critical} describes the situation where the norm of $ u $ does not decay at small parabolic scales, which causes a bottleneck in the analysis. They pointed out that the critical nature posed a major obstacle to proving a BV-type area change formula. The main theorem of the present note overcomes this difficulty and shows that the flow constructed in \cite{liu2024existence} satisfies in fact the following area change formula.

\begin{theorem}
\label{thm:main result}
	Let $ N \in \N $ with $ N \geq 2 $. Let $ \{ V_t \}_{ t \in \R^+ } $ be a Brakke's MCF with a critical forcing term $ u \in L^\infty_{loc} ( [ 0 , \infty ) ; L^2 ( \R^n ) ) \cap L^2_{loc} ( [ 0 , \infty ) ; W^{ 1 , 2 } ( \R^2 ) ) $ and let $ \{ E_i ( t ) \}_{ t \geq 0 } $ be a family of sets of finite perimeter satisfying $ \lV \nabla \chi_{ E_i ( t ) } \rV \leq \lV V_t \rV $ for each $ t \geq 0 $ and each $ i = 1 , \ldots , N $, constructed in \cite{liu2024existence}. Then, for any test function $ \phi \in C^1_c ( \R^2 \times [ 0 , \infty ) ) $ and $ 0 \leq t_1 < t_2 < \infty $, we have
\begin{equation}
\label{eq: BV_intro}
	\int_{ E_i ( t ) } \phi \, d x \bigg|_{ t = t_1 }^{ t_2 } = \int_{ t_1 }^{ t_2 } \int_{ E_i ( t ) } \partial_t \phi \, d x d t + \int_{ t_1 }^{ t_2 } \int_{ \R^2 } \phi ( h + u ) \cdot \nu_{ E_i ( t ) } \, d \lV \nabla \chi_{ E_i ( t ) } \rV d t,
\end{equation}
where $ h = h ( \cdot , V_t ) $ is the generalized mean curvature of $ V_t $, $ \lV \nabla \chi_{ E_i ( t ) } \rV $ is the perimeter measure of $ E_i ( t ) $, and $ \nu_{ E_i ( t ) } $ is the unit outer normal vector of $ \lV \nabla \chi_{ E_i ( t ) } \rV $. A detailed explanation of the flows by \cite{liu2024existence} will be provided in Section \ref{subsec: review}.
\end{theorem}
The equality \eqref{eq: BV_intro} was first identified in \cite{luckhaus1995implicit} as a characterization of weak MCF, called a BV flow. In the smooth case, \eqref{eq: BV_intro} holds naturally, while it is not obvious for generalized MCFs. Subsequently, Stuvard--Tonegawa \cite{StuvardTonegawa+2022} introduced the notion of generalized BV flow, which is a pair consisting of a Brakke flow and sets of finite perimeter $ E ( t ) $ satisfying \eqref{eq: BV_intro}, and they also proved the existence of generalized BV flow by refining Kim--Tonegawa's scheme \cite{kim2017mean}. Our result therefore shows that the flow constructed in \cite{liu2024existence} is not only a Brakke flow but also a generalized BV flow. Additionally, having \eqref{eq: BV_intro} has a certain conceptual advantage for the Brakke flow in that some non-uniqueness and stability issues of Brakke flow was resolved by Fischer--Hensel--Laux--Simon \cite{fischer2020local}.

To prove this result, we establish the following theorem, which provides minimal sufficient conditions for a Brakke flow (more precisely, an $ L^2 $ flow introduced in \cite{roger2008allen}) to satisfy the area change formula. It is worth emphasizing that the theorem imposes no restrictions on either the dimension or the velocity.
\begin{theorem}
	\label{thm:L2 flow area change}
	Let $ N \in \N $ with $ N \geq 2 $ and $ n \in \N $. Let $ \{ V_t \}_{ t \geq 0 } $ be an ($ n $-dimensional) $ L^2 $ flow defined in Definition \ref{def:L2 flow} with velocity $ v $ and $ \{ E_i ( t ) \}_{ t \geq 0 } $ be a family of sets of finite perimeter with $ \lV \nabla \chi_{ E_i ( t ) } \rV \leq \lV V_t \rV $ for all $ t \geq 0 $ and for each $ i = 1 , \ldots , N $. If the following three conditions hold for the pair $ ( \{ V_t \}_{ t \geq 0 } , \{ E_i ( t ) \}_{ t \geq 0 } ) $:
	\begin{enumerate}
		\item\label{itm:density assumption}
		$ \Theta^{ * n } ( x , \lV V_t \rV ) := \limsup_{ r \to + 0 } \lV V_t \rV ( B_r ( x ) ) / r^n < \infty $ for almost every $ t \geq 0 $ and for $ \lV V_t \rV $-almost every $ x \in \R^{ n + 1 } $,
		\item\label{itm:measure assumption}
		$ \lV ( \nabla , \partial_t ) \chi_{ E_i } \rV \ll \mu $ in $ \R^{ n + 1 } \times ( 0 , \infty ) $, where $ E_i := \{ ( x , t ) \in \R^{ n + 1 } \times ( 0 , \infty ) \mid x \in E_i ( t ) \} $ for each $ i = 1 , \ldots , N $, $ \lV ( \nabla , \partial_t ) \chi_{ E_i } \rV $ is the perimeter measure of $ \chi_{ E_i } $ and $ d \mu := d \lV V_t \rV d t $,
		\item\label{itm:continuity assumption}
		$ \chi_{ E_i } \in C^{ 1 / 2 }_{ loc } ( [ 0 , \infty ) ; L^1 ( \R^{ n + 1 } ) ) \cap BV_{ loc } ( \R^{ n + 1 } \times [ 0 , \infty ) ; \{ 0 , 1 \} ) $.
	\end{enumerate}
	Then $ v $ is an $ L^2_{ loc } ( \lV ( \nabla , \partial_t ) \chi_E \rV ; \R^{ n + 1 } ) $-function and we have
	\begin{equation}
		\label{eq:v area change formula}
		\begin{split}
			\int_{ E_i ( t ) } \phi \, d x \bigg\rv^{ t_2 }_{ t = t_1 } = \int_{ t_1 }^{ t_2 } \int_{ E_i ( t ) } \partial_t \phi \, d x d t
			+ \int_{ t_1 }^{ t_2 } \int_{ \partial^* E_i ( t ) } \phi ( v \cdot \nu_{ E_i ( t ) } ) \, d \mathcal{ H }^n d t
		\end{split}
	\end{equation}
	for all $ 0 \leq t_1 < t_2 < \infty $ and all test function $ \phi \in C^1_c ( \R^{ n + 1 } \times [ 0 , \infty ) ) $.
\end{theorem}
This theorem was implicitly established in \cite{StuvardTonegawa+2022,tashiro2024time} under a stronger assumption on the space-time density than that imposed in Theorem \ref{thm:L2 flow area change} (1). By refining their arguments, we weaken the density assumption and thereby make the theorem applicable to flows such as the one constructed in \cite{liu2024existence}. We also note that a compactness theorem for generalized BV flows can be derived by combining the present theorem with the argument developed in \cite{liu2024existence}. Although compactness is not the main focus of this note, we discuss this result in Section \ref{sec:compactness}. Roughly speaking, the compactness for the generalized BV flow is as follows:
\begin{theorem}
	Let $ ( \{ V_t^i \}_{ t \geq 0 } , \{ E_t^i \}_{ t \geq 0 } ) $ be a family of generalized BV flows satisfying
	\[
		\sup_i \sup_{ t \in [ 0 , T ] } \lV V_t^i \rV ( K ) \leq C
	\]
	for all $ T > 0 $ and all compact set $ K \subset \R^n $. Then, passing to a subsequence if necessary, there exists a generalized BV flow $ ( \{ V_t \}_{ t \geq 0 } , \{ E_t \}_{ t \geq 0 } ) $ such that $ ( \{ V_t^i \}_{ t \geq 0 } , \{ E_t^i \}_{ t \geq 0 } ) \to ( \{ V_t \}_{ t \geq 0 } , \{ E_t \}_{ t \geq 0 } ) $ in an appropriate sense.
\end{theorem}

As a corollary of Theorem \ref{thm:main result}, one can obtain the following on the estimate of the extinction time for the flow.

\begin{theorem}
	\label{thm: main2}
	Under the same setting as in Theorem \ref{thm:main result}, we let the extinction time of $ \{ V_t \}_{ t \geq 0 } $
	\[
		T^* := \sup \{ T \in [ 0 , \infty ] \mid \lV V_T \rV ( \R^2 ) \neq 0 \}.
	\]
	We may assume that $ \lV V_0 \rV = \mathcal{ H }^1 \lfloor_{ \R^2 \setminus \bigcup_{ i = 1 }^N E_i ( 0 ) } $ and $ N $ is the index of the only grain with infinite volume without loss of generality. Define the bounded open set $ E ( t ) := \bigcup_{ i = 1 }^{ N - 1 } E_i ( t ) $. Then $ T^* $ satisfies
	\begin{equation}
		\label{eq: esti extinction}
		\frac{ 2 \lv E ( 0 ) \rv^2 }{ \lV V_0 \rV ( \R^2 )^2 ( 1 + C ( u ) ) } \leq T^*
	\end{equation}
	for some $ 0 \leq C ( u ) < \infty $.
\end{theorem}

In the case of $ u \equiv 0 $, the similar estimate on extinction time \eqref{eq: esti extinction} was obtained by Giga--Yama-uchi \cite{giga1993lower} for level set solutions and Salvatore--Tonegawa \cite{StuvardTonegawa+2022} for multi-phase generalized BV flows, and their estimates are sharp. However, as far as the author knows, such an estimate had never been proved before in the context of MCFs with (critical) forcing term. The estimate \eqref{eq: esti extinction} ensures that extinction does not occur until the time provided by this theorem even if one chooses $ u $ that tends to contract the domain $ E ( t ) $. Note that the extinction time of a MCF with forcing term can be infinite.

In the subcritical case, namely, when the influence of a forcing term is negligible at small parabolic scales, there are several previous studies using the Allen--Cahn equation and we mention \cite{liu2010existence,takasao2016existence}. Moreover, the author showed the flows constructed in \cite{liu2010existence,takasao2016existence} satisfy the same result as \eqref{eq: BV_intro} in \cite{tashiro2024time}.

The proof of \eqref{eq: BV_intro} mainly follows \cite{StuvardTonegawa+2022}. The difficulty to prove \eqref{eq: BV_intro} is a lack of Huisken's monotonicity formula for the MCF with critical forcing term $ u $. In \cite{StuvardTonegawa+2022}, the monotonicity formula is essentially used to obtain the bound of the space-time density of Brakke flows,
which implies that the absolute continuity of $ d \lV V_t \rV d t $ with respect to the Hausdorff measure $ \mathcal{ H }^2 $ in $ \R^2 \times ( 0 , \infty ) $. Due to the criticality, we do not know this absolute continuity. To overcome this, we use the measure theoretic properties between $ \lV \nabla \chi_{ E ( t ) } \rV $ and $ \lV V_t \rV $ and a priori estimate on the density of $ \lV V_t \rV $ by \cite{liu2024existence} to obtain $ ( d \lV V_t \rV d t ) \lfloor_{ \partial S } \ll d \mathcal{ H }^2 \lfloor_{ \partial S } $ instead of $ d \lV V_t \rV d t \ll d \mathcal{ H }^2 $, where $ S := \{ ( x , t ) \mid x \in E ( t ) \} $. We then show that this weak property gives the area change formula, summarized in Theorem \ref{thm:L2 flow area change}.

The paper is organized as follows: In Section \ref{sec: pre}, we set our notation and review the results of \cite{liu2024existence}. In Section \ref{sec:proof of area change}, we show the proof of Theorem \ref{thm:L2 flow area change} under a weak assumption. In Section \ref{sec: proof of main}, we show the absolute continuity between a Brakke flow $ d \lV V_t \rV d t $ and $ \mathcal{ H }^2 \lfloor_{ \partial S } $, and we then prove that the flow by \cite{liu2024existence} satisfies the area change formula \eqref{eq: BV_intro}. In Section \ref{sec:compactness}, we briefly give a proof of the compactness theorem for generalized BV flows. In Section \ref{sec: proof of main2}, we estimate the extinction time.

\subsection*{Acknowledgments}
The author would like to thank Prof. Yoshihiro Tonegawa for discussing the problem of this paper. The author also acknowledges support from Japan Society for Promotion of Science (JSPS) through the Grand-in-Aid for JSPS Fellows, Grant number 25KJ1245, as well as support from the Italian Ministry of University and Research (MUR) through the FIS project \textit{SiGmA: ``Singularities in Geometric Analysis: Minimal Surfaces and Mean Curvature Flows''}, project code FIS-2023-02962 (CUP G53C25000120001). 

\section{Preliminaries}
\label{sec: pre}

\subsection{Basic Notation}
\label{subsec: notation}
We shall use the same notation for the most part adopted in \cite{liu2024existence}. For any two Radon measures $ \mu $ and $ \nu $, $ \mu \ll \nu $ denotes that $ \mu $ is absolutely continuous with respect to $ \nu $. A general $ n $-varifold is a Radon measure on $ \R^{ n + 1 } \times \mathbf{ G } ( n + 1 , n ) $, where $ \mathbf{ G } ( n + 1 , n ) $ is the space of $ n $-dimensional subspace of $ \R^{ n + 1 } $ (see \cite{allard1972first,simon1983lectures} for the details).  The set of all general $ n $-varifolds is denoted by $ \V_n ( \R^{ n + 1 } ) $. For $ V \in \V_n ( \R^{ n + 1 } ) $, we let $ \lV V \rV $ the weight measure of $ V $, namely,
\[
\lV V \rV ( \phi ) := \int_{ \R^{ n + 1 } \times \mathbf{ G } ( n + 1 , n ) } \phi \, d V ( x , S ) \text{ for } \phi \in C^1_c ( \R^{ n + 1 } ).
\]
We call $ V \in \V_n ( \R^{ n + 1 } ) $ rectifiable if there exist an $ \mathcal{ H }^n $-measurable countably $ n $-rectifiable set $ M \subset \R^{ n + 1 } $ (see \cite{simon1983lectures} for the definition) and a locally $ \mathcal{ H }^n $-integrable function $ \theta $ defined on $ M $ such that
\[
	V ( \phi ) = \int_M \phi ( x , \Tan_x M ) \theta ( x ) \, d \mathcal{ H }^n ( x ) \text{ for } \phi \in C_c ( \R^{ n + 1 } \times \mathbf{ G } ( n + 1 , n ) ).
\]
Here $ \Tan_x M $ is the approximate tangent space of $ M $ at $ x $ which exists $ \mathcal{ H }^n $-almost everywhere on $ M $. The approximate tangent space $ \Tan_x M $ is denoted in this case by $ \Tan_x \lV V \rV $ without fear of confusion. If $ \theta \in \N $ $ \mathcal{ H }^n $-almost everywhere on $ M $, we say $ V $ is integral. The set of all integral $ n $-varifolds is denoted by $ \IV_n ( \R^{ n + 1 } ) $. If $ \theta = 1 $ $ \mathcal{ H }^n $-almost everywhere on $ M $, we say $ V $ is a unit density $ n $-varifold.

For $ V \in \V_n ( \R^{ n + 1 } ) $, let $ \delta V $ be the first variation of $ V $, namely,
\[
	\delta V ( g ) := \int_{ \R^{ n + 1 } \times \mathbf{ G } ( n + 1 , n ) } \nabla g ( x ) \cdot S \, d V ( x , S ) \text{ for } g \in C^1_c ( \R^{ n + 1 } ; \R^{ n + 1 } ),
\]
where $ \nabla g \cdot S = \sum_{ i , j = 1 }^{ n + 1 } \partial_{ x_j } g^i S_{ i j } $, which is the divergence of $ g $ on $ S $. If
\[
	\sup_{ g \in C^1_c ( \R^{ n + 1 } ; \R^{ n + 1 } ), \lv g \rv \leq 1 } \lv \delta V ( g ) \rv < \infty,
\]
we say the first variation is bounded, and in this case, $ \lV \delta V \rV $ denotes the total variation measure of $ \delta V $. If $ \lV \delta V \rV $ is absolutely continuous with respect to $ \lV V \rV $, we have a $ \lV V \rV $-measurable vector field $ h ( \cdot , V ) $ such that
\[
	\delta V ( g ) = - \int_{ \R^{ n + 1 } } g ( x ) \cdot h ( x , V ) \, d \lV V \rV ( x ) \text{ for } g \in C^1_c ( \R^{ n + 1 } ; \R^{ n + 1 } ).
\]
The vector field $ h ( \cdot , V ) $ is called the generalized mean curvature vector of $ V $. For any $ V \in \IV_n ( \R^{ n + 1 } ) $ with bounded first variation, by Brakke's perpendicularity theorem \cite[Chapter 5]{brakke1978motion}, one has
\[
	S^{ \perp } ( h ( x , V ) ) = h ( x , V ) \text{ for } V \text{-} a.e.\, ( x , S ) \in \R^{ n + 1 } \times \mathbf{ G } ( n + 1 , n ),
\]
that is, $ h ( x , V ) $ is perpendicular to the approximate tangent space for $ \lV V \rV $-almost every $ x \in \R^{ n + 1 } $. For $ V \in \IV_n ( \R^{ n + 1 } ) $, and $ \lV V \rV $-integrable function $ u $, we use the notation
\[
	u^{ \perp } ( x ) := ( \Tan_x \lV V \rV )^{ \perp } ( u ( x ) ),
\]
which is well-defined for $ \lV V \rV $-almost every $ x \in \R^{ n + 1 } $.

For a set of finite perimeter $ E \subset \R^{ n + 1 } $, let $ \partial^* E $ be the reduced boundary of $ E $ and $ \nu_E $ be the outer unit normal vector of $ \partial^* E $. Let $ \nabla \chi_E $ be the distributional derivative of characteristic function of $ E $, and let $ \lV \nabla \chi_E \rV $ be the total variation of $ \nabla \chi_E $. We also let $ \nabla' = ( \nabla , \partial_t ) $, which is the full divergence in $ \R^{ n + 1 } \times \R $. By the well-known structure theorem for set of finite perimeter, one has $ \lV \nabla \chi_E \rV = \mathcal{ H }^n \lfloor_{ \partial^* E } $ (see \cite{maggi2012sets} for the details on the set of finite perimeter, for example). We also write $ \lv E \rv $ as the Lebesgue measure of $ E $.

\subsection{Review on Liu--Tonegawa's Result}
\label{subsec: review}
In this subsection, we briefly recall the results what we will use from \cite{liu2024existence} for the convenience of the readers (see also \cite{kim2017mean,kim2020existence,StuvardTonegawa+2022}). We set an initial datum satisfying the following assumption.

\begin{assumption}
\label{assum: initial}
	As an initial datum, we consider the following:
	\begin{enumerate}
		\item $ N \in \N $ with $ N \geq 2 $.
		\item $ E_{ 0 , 1 } , \ldots , E_{ 0 , N } $ are non-empty open sets in $ \R^2 $ which are mutually disjoint.
		\item $ \Gamma_0 := \R^2 \setminus \bigcup_{ i = 1 }^N E_{ 0 , i } $ is countably $ 1 $-rectifiable and $ \mathcal{ H }^1 ( \Gamma_0 ) < \infty $.
		\item $ \mathcal{ H }^1 ( \bigcup_{ i = 1 }^N ( \partial E_{ 0 , i } \setminus \partial^* E_{ 0 , i } ) ) = 0 $. 
	\end{enumerate}
\end{assumption}

The above $ N $ corresponds to the number of phases (or grains) $ E_{ 0 , 1 } , \ldots , E_{ 0 , N } $ and they divide the whole space $ \R^2 $. Since the MCF originally arises from a mathematical model of the metal annealing, considering the multi phase case is natural. Moreover, since $ \Gamma_0 $ do not have any interior point, we have $ \Gamma_0 = \bigcup_{ i = 1 }^N \partial E_{ 0 , i } $ and the initial datum $ \Gamma_0 $ is given as the topological boundary of partition.

For any smooth forcing term $ u \in C^1_c ( \R^2 \times [ 0 , \infty ) ; \R^2 ) $, Liu--Tonegawa established the existence of a MCF with $ u $ by modifying the scheme of \cite{kim2017mean,kim2020existence,StuvardTonegawa+2022} and we quote the results in \cite[Section 4 and Appendix A]{liu2024existence}. Note that they proved this in the higher dimension case.

\begin{theorem}
\label{thm: existence for smooth u}
	Let $ \Gamma_0 $ and $ \{ E_{ 0 , i } \}_{ i = 1 }^N $ be as in Assumption \ref{assum: initial} and let $ u \in C^1_c ( \R^2 \times [ 0 , \infty ) ; \R^2 ) $. Then there exist a family $ \{ V_t \}_{ t \geq 0 } \subset \IV_1 ( \R^2 ) $ and a family of open sets $ \{ E_i ( t ) \}_{ t \geq 0 } $ for each $ i = 1 , \ldots , N $ with the following properties:
	\begin{enumerate}
		\item $ \lim_{ t \to + 0 } \lV V_t \rV = \lV V_0 \rV = \mathcal{ H }^1 \lfloor_{ \Gamma_0 } $.
		\item For almost every $ t \geq 0 $, $ \delta V_t $ is bounded and absolutely continuous with respect to $ \lV V_t \rV $ (thus the generalized mean curvature $ h ( \cdot , V_t ) $ exists).
		\item For all $ T > 0 $, $ \{ V_t \}_{ t \geq 0 } $ and $ h ( \cdot , V_t ) $ satisfy
		\begin{equation*}
			\sup_{ 0 \leq t \leq T } \lV V_t \rV ( \R^2 ) < \infty \text{ and } \int_0^T \int_{ \R^2 } \lv h ( \cdot , t ) \rv^2 \, d \lV V_t \rV d t < \infty.
		\end{equation*}
		\item For all $ 0 \leq t_1 < t_2 < \infty $ and $ \phi \in C^1_c ( \R^2 \times [ 0 , \infty ) ) $ with $ \phi \geq 0 $, we have
		\begin{equation}
			\label{eq: Brakke}
				\lV V_t \rV ( \phi ( \cdot , t ) ) \Big|_{ t = t_1 }^{ t_2 } \leq \int_{ t_1 }^{ t_2 } \int_{ \R^2 } ( \nabla \phi - \phi h ) \cdot ( h + u^{ \perp } ) + \partial_t \phi \, d \lV V_t \rV d t.
		\end{equation}
		\item For $ i = 1 , \ldots , N $, $ E_i ( 0 ) = E_{ 0 , i } $ and $ E_1 ( t ) , \ldots , E_N ( t ) $ are mutually disjoint open sets for all $ t \geq 0 $.
		\item For all $ t \geq 0 $ and $ i = 1 , \ldots , N $, $ \lV \nabla \chi_{ E_i ( t ) } \rV \leq \lV V_t \rV $.
		\item $ S ( i ) := \{ ( x , t ) \in \R^2 \times [ 0 , \infty ) \mid x \in E_i ( t ) \} $ is open in $ \R^2 \times [ 0 , \infty ) $ for $ i = 1 , \ldots , N $.
		\item For all $ T > 0 $, there exists a constant $ C = C ( u , T , \Gamma_0 ) $ such that for each $ i = 1 , \ldots , N $ and $ 0 \leq t_1 < t_2 \leq T $,
		\[
			\lv ( E_i ( t_1 ) \setminus E_i ( t_2 ) ) \cup ( E_i ( t_2 ) \setminus E_i ( t_1 ) ) \rv \leq C ( t_2 - t_1 )^{ 1 / 2 }
		\]
		\item For all $ 0 \leq t_1 < t_2 < \infty $, $ \phi \in C^1_c ( \R^2 \times [ 0 , \infty ) ) $ and $ i = 1 , \ldots , N $, we have
		\begin{equation}
		\label{eq: BV u is smooth}
			\int_{ E_i ( t ) } \phi ( \cdot , t ) \, d x \bigg|_{ t = t_1 }^{ t_2 } = \int_{ t_1 }^{ t_2 } \int_{ E_i ( t ) } \partial_t \phi \, d x d t + \int_{ t_1 }^{ t_2 } \int_{ \R^2 } \phi ( h + u ) \cdot \nu_{ E_i ( t ) } \, d \lV \nabla \chi_{ E_i ( t ) } \rV d t.
		\end{equation}
	\end{enumerate}
\end{theorem}

For a given vector field $ u $ as in \eqref{eq: critical} and $ m \in \N $, we extend $ u $ by $ 0 $ for $ t \notin [ 0 , m ] $ and we set the appropriately mollified vector field $ u^{ ( m ) } $. Let $ \{ V^{ ( m ) }_t \}_{ t \geq 0 } $ and $ \{ E_i^{ ( m ) } ( t ) \}_{ t \geq 0 } $ $ ( i = 1 , \ldots , N ) $ be the flow corresponding to $ u = u^{ ( m ) } $ in Theorem \ref{thm: existence for smooth u} and consider the limit. By the density estimate by \cite{pozzetta2020varifold} and the Meyers--Ziemer inequality \cite{meyers1977integral}, Liu--Tonegawa established the following a priori estimates (\cite[Proposition 3.3]{liu2024existence}).
\begin{prop}
\label{prop: a priori}
	Let $ u \in C^1_c ( \R^2 \times [ 0 , \infty ) ; \R^2 ) $ and let $ \{ V_t \}_{ t \geq 0 } $ be as in Theorem \ref{thm: existence for smooth u}. Then, there is an absolute constant $ c > 0 $ such that, for all $ T > 0 $, we have
	\begin{align}
		\sup_{ 0 \leq t \leq T } \lV V_t \rV ( \R^2 ) &\leq \lV V_0 \rV ( \R^2 ) \exp ( c^2 C ) \label{eq: mass a priori},\\
		\sup_{ B_r ( x ) \subset \R^2 } \frac{ \lV V_t \rV ( B_r ( x ) ) }{ r } &\leq \big( \lV V_t \rV ( \R^2 ) \big)^{ 1 / 2 } \bigg( \int_{ \R^2 } \lv h ( \cdot , V_t ) \rv^2 \, d \lV V_t \rV \bigg)^{ 1 / 2 } \ (\text{for } a.e.\, t \geq 0 ) \label{eq: density_esti},\\
		\int_0^T \int_{ \R^2 } \lv h ( \cdot , V_t ) \rv^2 \, d \lV V_t \rV d t &\leq 4 \lV V_0 \rV ( \R^2 ) ( 1 + c^2 C \exp ( c^2 C ) ) \label{eq: h a priori},\\
		\int_0^T \int_{ \R^2 } \lv u \rv^2 \, d \lV V_t \rV d t &\leq 4 c \sqrt{ C } \lV V_0 \rV ( \R^2 ) \sqrt{ 1 + c^2 C \exp ( c^2 C ) } \exp ( c^2 C / 2 ) \label{eq: u a priori},
	\end{align}
	where
	\[
		C = C ( u , T ) := \ess_{ 0 \leq t \leq T }\,\int_{ \R^2 } \lv u ( x , t ) \rv^2 \, d x \times \int_0^T \int_{ \R^2 } \lv \nabla u ( x , t ) \rv^2 \, d x d t.
	\]
\end{prop}

Liu--Tonegawa showed that a solution of the MCF with $ u $ can be obtained as a suitable limit of $ V_t^{ ( m ) } $, $ E_i^{ ( m ) } ( t ) $ and $ u^{ ( m ) } $ by the above a priori estimates, and we quote \cite[Theorem 2.2]{liu2024existence} here. Note that they obtained various additional properties of the flow not mentioned here.

\begin{theorem}
\label{thm: existence for critical u}
	Let $ \Gamma_0 $ and $ \{ E_{ 0 , i } \}_{ i = 1 }^N $ be as in Assumption \ref{assum: initial} and let $ u \in L^{ \infty }_{ loc } ( [ 0 , \infty ) ; L^2 ( \R^2 ) ) \cap L^2_{ loc } ( [ 0 , \infty ) ; W^{ 1 , 2 } ( \R^2 ) ) $. Then there exist a family $ \{ V_t \}_{ t \geq 0 } \subset \IV_1 ( \R^2 ) $ and a family of open sets $ \{ E_i ( t ) \}_{ t \geq 0 } $ for each $ i = 1 , \ldots , N $ with the following properties:
	\begin{enumerate}
		\item $ \{ V_t \}_{ t \geq 0 } $ and $ \{ E_i ( t ) \}_{ t \geq 0 } $ satisfy the same properties as in Theorem \ref{thm: existence for smooth u} (1)-(8).
		\item For all $ T > 0 $, $ V_t $ and $ u $ satisfy the same a priori estimates as in Proposition \ref{prop: a priori}.
	\end{enumerate}
\end{theorem}

\section{Proof of Theorem \ref{thm:L2 flow area change}}
\label{sec:proof of area change}
Let $ \{ V_t \}_{ t \geq 0 } $, $ \{ E_i ( t ) \}_{ t \geq 0 } $ as in Theorem \ref{thm:L2 flow area change}. Throughout this section, for simplicity, we assume $ N = 2 $ and write $ E_1 ( t ) = E_t $ and $ E_2 ( t ) = \R^{ n + 1 } \setminus E_1 ( t ) $ without loss of generality. Since Theorem \ref{thm:L2 flow area change} is proved independently of the result in \cite{liu2024existence}, no restriction on the dimension is required. We let $ n \in \N $, which describes the dimension of the flow.

For the purpose of this paper, we introduce the following notion of velocity for a family of varifolds.

\begin{definition}
	\label{def:L2 flow}
	A family of varifolds $ \{ V_t \}_{ t \geq 0 } $ is \textit{an} $ n $-\textit{dimensional $ L^2 $ flow with velocity $ v $} if it satisfies the following:
	\begin{description}
		\item[\textup{(a)}] For almost every $ t \geq 0 $, $ V_t \in \IV_n ( \R^{ n + 1 } ) $ and $ V_t $ has its generalized weak mean curvature $ h ( \cdot , V_t ) $.
		\item[\textup{(b)}] $ d \mu := d \lV V_t \rV d t $ is a Radon measure on $ \R^{ n + 1 } \times [ 0 , \infty ) $.
		\item[\textup{(c)}] There exists a vector field $ v \in L^2_{ loc } ( \mu ; \R^{ n + 1 } ) $ such that
		\begin{description}
			\item[\textup{(c'1)}] $ v ( x , t ) \perp T_x \, \lV V_t \rV $ for $ \mu $-almost every $ ( x , t ) \in \R^{ n + 1 } \times [ 0 , \infty ) $,
			\item[\textup{(c'2)}] For every test functions $ \phi \in C^1_c ( \R^{ n + 1 } \times [ 0 , \infty ) ) $, it holds that
			\begin{equation}
				\label{eq:L2 flow}
				\Bigg\lv \int_0^{ \infty } \int_{ \R^{ n + 1 } } \partial_t \phi + \nabla \phi \cdot v \, d \lV V_t \rV d t  \Bigg\rv \leq C \sup \lv \phi \rv
			\end{equation}
			for some $ C = C ( \{ V_t \}_{ t \geq 0 } , v , \spt \phi ) > 0 $.
		\end{description}
	\end{description}
\end{definition}

This is well-known as the $ L^2 $ flow, a characterization of the weak velocity of varifolds proposed by \cite{roger2008allen}. The following is a simple corollary of the above proposition (\cite[Proposition 3.3]{roger2008allen}).

\begin{prop}
	\label{prop: L2 cor}
	It holds that
	\[
	\begin{pmatrix}
		v ( x , t )\\
		1
	\end{pmatrix}
	\in \Tan_{ ( x , t ) } \mu
	\]
	at $ \mu $-almost every $ ( x , t ) $ whenever the approximate tangent space $ \Tan_{ ( x , t ) } \mu $ exists.
\end{prop}

From the measure-theoretic inequality $ \lV \nabla \chi_{ E_t } \rV \leq \lV V_t \rV $, we have the space-time absolute continuity between $ \lV \nabla^{ \prime } \chi_E \rV $ and $ \mu $ as follows.

\begin{lemma}
	\label{lem: V<<S}
	Let $ \{ V_t \}_{ t \geq 0 } $ and $ \{ E_t \}_{ t \geq 0 } $ satisfy all the conditions of Theorem \ref{thm:L2 flow area change}. Then we have $ \mu \lfloor_{ \partial^* E } \ll \lV \nabla^{ \prime } \chi_E \rV $ in $ \R^{ n + 1 } \times ( 0 , \infty ) $.
\end{lemma}
\begin{proof}
	First, we note that $ d \mu \lfloor_{ \partial^* E } = d \lV V_t \rV \lfloor_{ ( \partial^* E )_t } d t $, where
	\[
		( \partial^* E )_t := \{ x \in \R^{ n + 1 } \mid ( x , t ) \in \partial^* E \}.
	\]
	Since it follows from the co-area formula (see \cite[Theorem 13.1]{maggi2012sets}) that
	\[
		\int_0^{ \infty } \int_{ \R^{ n + 1 } } \phi \, d \lV \nabla \chi_{ E_t } \rV d t \leq \int_{ \R^{ n + 1 } \times ( 0 , \infty ) } \phi \, d \lV \nabla^{ \prime } \chi_E \rV
	\]
	for all $ \phi \in C_c ( \R^{ n + 1 } \times ( 0 , \infty ) ) $ with $ \phi \geq 0 $, if we prove $ d \lV V_t \rV \lfloor_{ ( \partial^* E )_t } d t \ll d \lV \nabla \chi_{ E_t } \rV d t $, one can obtain the desired absolute continuity
	\[
		d \mu \lfloor_{ \partial^* E } = d \lV V_t \rV \lfloor_{ ( \partial^* E )_t } d t \ll d \lV \nabla \chi_{ E_t } \rV d t \ll d \lV \nabla^{ \prime } \chi_E \rV \text{ in } \R^{ n + 1 } \times ( 0 , \infty ).
	\]
	Next, we see that $ \lV V_t \rV \ll \mathcal{ H }^n $ holds for almost every $ t \geq 0 $. Indeed, let $ A \subset \R^{ n + 1 } $ be a set with $ \mathcal{ H }^n ( A ) = 0 $ and let
	\[
		D_k := \{ x \in \R^{ n + 1 } \mid \Theta^{ * n } ( x , \lV V_t \rV )\leq k \}.
	\]
	For almost every $ t \geq 0 $, the upper density $ \Theta^{ * n } ( x , \lV V_t \rV ) $ is finite by assumption in Theorem \ref{thm:L2 flow area change} \eqref{itm:density assumption} for $ \lV V_t \rV $-almost every $ x \in \R^{ n + 1 } $. Let such a $ t \geq 0 $ fix. From \cite[Theorem 3.2]{simon1983lectures}, it follows that $ \lV V_t \rV ( A \cap D_k ) \leq 2 k \mathcal{ H }^n ( A \cap D_k ) = 0 $ for all $ k \in \N $. By the finiteness of the upper density of $ \lV V_t \rV $, we have $ \lV V_t \rV ( A \setminus \bigcup_{ k \in \N } D_k ) = 0 $. Therefore, we obtain $ \lV V_t \rV ( A ) = 0 $, which implies $ \lV V_t \rV \ll \mathcal{ H }^n $ for almost every $ t \geq 0 $. Since
	\[	
	\mathcal{ H }^n ( ( ( \partial^* E )_t \setminus \partial^* E_t ) \cup ( \partial^* E_t \setminus ( \partial^* E )_t ) ) = 0 \text{ for } a.e.\,t \geq 0
	\]
	holds by the general theory of the set of finite perimeter (see \cite[Theorem 18.11]{maggi2012sets}, for example) and $ \lV V_t \rV \ll \mathcal{ H }^n $ for almost every $ t \geq 0 $, we have
	\[
	\lV V_t \rV \lfloor_{ ( \partial^* E )_t } = \lV V_t \rV \lfloor_{ \partial^* E_t } \text{ for } a.e.\,t \geq 0.
	\]
	Let $ A \subset \R^{ n + 1 } \times ( 0 , \infty ) $ be a set with $ \int \int_A d \lV \nabla \chi_{ E_t } \rV d t = 0 $. Since
	\[
	0 = \int \int_A d \lV \nabla \chi_{ E_t } \rV d t = \int_0^{ \infty } \int_{ A_t } d \lV \nabla \chi_{ E_t } \rV d t
	\]
	holds, where $ A_t = \{ x \in \R^{ n + 1 } \mid ( x , t ) \in A \} $, we particularly have $ \lV \nabla \chi_{ E_t } \rV ( A_t ) = 0 $ for almost every $ t > 0 $. By the absolute continuity $ \lV V_t \rV \ll \mathcal{ H }^n $, we obtain $ \lV V_t \rV \lfloor_{ \partial^* E_t } ( A_t ) = 0 $ for almost every $ t > 0 $, which gives $ d \lV V_t \rV \lfloor_{ \partial^* E_t } d t \ll d \lV \nabla \chi_{ E_t } \rV d t $. This completes the proof.
\end{proof}

\begin{remark}
	With respect to assumption \eqref{itm:density assumption} in Theorem \ref{thm:L2 flow area change}, one can see that the space-time density of the Brakke flow is finite almost everywhere because of the argument from Huisken's monotonicity formula (see \cite[Proposition 3.5]{tonegawa2019brakke}), which leads to the stronger measure-theoretic continuity $ \mu \ll \mathcal{ H }^{ n + 1 } $ in $ \R^{ n + 1 } \times ( 0 , \infty ) $  than Lemma \ref{lem: V<<S}. Stuvard and Tonegawa used the stronger property to prove the existence of the generalized BV flow in \cite{StuvardTonegawa+2022}. However, this density argument is not applicable to the flows which no Huisken-type monotonicity formula is known, such as the flow from \cite{liu2024existence}. We thus derive the area change formula \eqref{eq:v area change formula} from the weak continuity in Lemma \ref{lem: V<<S} only.
\end{remark}

The key step of the proof of Theorem \ref{thm:L2 flow area change} is to prove the following proposition, for which the property of $ L^2 $ flow plays a central role.

\begin{prop}
	\label{prop: basic}
	The restricted Radon measure $ \mu \lfloor_{ \partial^* E } $ is ($ n + 1 $)-rectifiable and we have the following for $ \mathcal{ H }^{ n + 1 } $-almost every $ ( x , t ) \in \partial^* E \cap \{ t > 0 \} $:
	\begin{enumerate}
		\item the approximate tangent space $ T_{ ( x , t ) } \, \mu $ exists and $ T_{ ( x , t ) } \, \mu = T_{ ( x , t ) } \, ( \partial^* E ) $,
		\item $ {}^t ( v ( x , t ) , 1 ) \in T_{ ( x , t ) } \, \mu $,
		\item $ x \in \partial^* E $ and $ T_x \, \lV V_t \rV = T_x \, ( \partial^* E_t ) $,
		\item $ \mathbf{ p } ( \nu_{ E } ) \neq 0 $ and $ \nu_{ E_t } ( \cdot ) = \lv \mathbf{ p } ( \nu_{ E } ( \cdot , t ) ) \rv^{ - 1 } \mathbf{ p } ( \nu_{ E } ( \cdot , t ) ) $, where $ \mathbf{ p } $ denotes the projection of $ \R^{ n + 1 } \times \R $ onto $ \R^{ n + 1 } $,
		\item $ T_x \, ( \partial^* E_t ) \times \{ 0 \} $ is the linear subspace of $ T_{ ( x , t ) } \, \mu $.
	\end{enumerate}
\end{prop}
\begin{proof}
	By assumption \eqref{itm:measure assumption} in Theorem \ref{thm:L2 flow area change} and Lemma \ref{lem: V<<S}, we see that
	\[
		\lV \nabla' \chi_E \rV \ll \mu \lfloor_{ \partial^* E } \text{ and } \mu \lfloor_{ \partial^* E } \ll \lV \nabla' \chi_E \rV.
	\]
	By the Radon--Nykod\'{y}m theorem, there exists a function
	\[
		f = \frac{ \mu \lfloor_{ \partial^* E } }{ \lV \nabla' \chi_E \rV } \in L^1 ( \lV \nabla' \chi_E \rV )
	\]
	such that $ \mu \lfloor_{ \partial^* E } = f \, \lV \nabla^{ \prime } \chi_E \rV $ with $ 0 < f < \infty $ $ \mathcal{ H }^{ n + 1 } $-almost everywhere, which shows that $ \mu \lfloor_{ \partial^* E } $ is ($ n + 1 $)-rectifiable.
	
	We next prove that $ T_{ ( x , t ) } \, \mu = T_{ ( x , t ) } \, ( \partial^* E ) $ for $ \mathcal{ H }^{ n + 1 } $-almost every $ ( x , t ) \in \partial^* E \cap \{ t > 0 \} $. By \cite[Theorem 3.5]{simon1983lectures}, one sees that
	\[
	\limsup_{ r \to + 0 } \frac{ \mu ( B_r ( x , t ) \setminus \partial^* E ) }{ r^{ n + 1 } } = 0 \quad \text{ for } \mathcal{ H }^{ n + 1 } \text{-} a.e. \, ( x , t ) \in \partial^* E \cap \{ t > 0 \}.
	\]
	Let $ \phi \in C_c ( B_1 ( 0 , 0 ) ) $ be arbitrary. We then have
	\begin{equation*}
		\begin{split}
			\lim_{ r \to + 0 } &\bigg\lv \int_{ \R^{ n + 1 } \times ( 0 , \infty ) \setminus \partial^* E } \frac{ 1 }{ r^{ n + 1 } } \phi \bigg( \frac{ y - x }{ r } , \frac{ s - t }{ r } \bigg) \, d \mu ( y , s ) \bigg\rv\\
			&\leq \sup_{ B_1 ( 0 , 0 ) } \lv \phi \rv \limsup_{ r \to + 0 } \frac{ \mu ( B_r ( x , t ) \setminus \partial^* E ) }{ r^{ n + 1 } } = 0
		\end{split}
	\end{equation*}
	for $ \mathcal{ H }^{ n + 1 } $-almost every $ ( x , t ) \in \partial^* E \cap \{ t > 0 \} $. Thus, at each Lebesgue point of $ f $, we obtain
	\[
	\lim_{ r \to + 0 } \int_{ \R^{ n + 1 } \times ( 0 , \infty ) } \frac{ 1 }{ r^{ n + 1 } } \phi \bigg( \frac{ y - x }{ r } , \frac{ s - t }{ r } \bigg) \, d \mu ( y , s ) = f ( x , t ) \int_{ T_{ ( x , t ) } \, ( \partial^* E ) } \phi \, d \mathcal{ H }^{ n + 1 },
	\]
	which completes the proof of $ T_{ ( x , t ) } \, \mu = T_{ ( x , t ) } \, ( \partial^* E ) $ by $ f \in L^1 ( \lV \nabla^{ \prime } \chi_E \rV ) $. Therefore (1) is satisfied. By Proposition \ref{prop: L2 cor}, we have (2).

	Next, we show that parts (3) and (4). By the standard theorem of set of finite perimeter (see \cite[Theomre 18.11]{maggi2012sets}), one sees the following for almost every $ t > 0 $ and $ \mathcal{ H }^n $-almost every $ x \in ( \partial^* E_t ) $:
	\begin{align}
		\mathcal{ H }^n ( ( ( \partial^* E )_t \setminus \partial^* E_t ) \cup ( \partial^* E_t \setminus ( \partial^* E )_t ) ) &= 0, \label{eq: 1}\\
		\mathbf{ p } ( \nu_{ E } ( x , t ) ) &\neq 0, \label{eq: 2}\\
		\nu_{ E_t } ( x ) &= \lv \mathbf{ p } ( \nu_{ E } ( x , t ) ) \rv^{ - 1 } \mathbf{ p } ( \nu_{ E } ( x , t ) ). \label{eq: 3}
	\end{align}
	We let $ I := \{ t > 0 \mid \eqref{eq: 1} \text{ fails} \} $ and $ A_t := \{ x \in ( \partial^* E )_t \mid x \notin \partial^* E_t \text{ or } \eqref{eq: 2} \text{ and } \eqref{eq: 3} \text{ fail} \} $ for every $ t > 0 $ so that $ \lv I \rv = 0 $ and $ \mathcal{ H }^n ( A_t ) = 0 $ for all $ t \in ( 0 , \infty ) \setminus I $. Set the characteristic function $ \chi ( x , t ) := \chi_{ A_t } ( x ) $ on $ \R^{ n + 1 } \times ( 0 , \infty ) $. By the coarea formula, we then have
	\begin{equation*}
		\begin{split}
			&\int_{ \partial^* E } \chi ( x , t ) \lv \nabla^{ \partial^* E } ( \mathbf{ q } ( x , t ) ) \rv \, d \mathcal{ H }^{ n + 1 } ( x , t ) = \int_0^{ \infty } \int_{ ( \partial^* E )_t } \chi ( x , t ) \, d \mathcal{ H }^n ( x ) d t\\
			&= \int_0^{ \infty } \mathcal{ H }^n ( A_t ) \, d t = \int_I \mathcal{ H }^n ( A_t ) \, d t = 0,
		\end{split}
	\end{equation*}
	where $ \nabla^{ \partial^* E } ( \cdot ) = P_{ T_{ ( x , t ) } ( \partial^* E ) } ( \nabla ( \cdot ) ) $ and $ \mathbf{ q } $ denotes the projection of $ \R^{ n + 1 } \times \R $ onto $ \R $. Combining (1) and (2) implies that $ \lv \nabla^{ \partial^* E } ( \mathbf{ q } ( x , t ) ) \rv > 0 $ for $ \mathcal{ H }^{ n + 1 } $-almost every $ ( x , t ) \in \partial^* E \cap \{ t > 0 \} $. Hence, we have $ \chi ( x , t ) = 0 $ for $ \mathcal{ H }^{ n + 1 } $-almost every $ ( x , t ) \in \partial^* E \cap \{ t > 0 \} $. Thus, the first part of (3) is proved and we have (4). One can prove the identity $ T_x \, \lV V_t \rV = T_x \, ( \partial^* E_t ) $ by assumptions that and repeating the proof of part (1) at fixed $ t $.
	
	Finally, taking $ ( x , t ) \in \partial^* E $ as satisfying (1)-(4) of this proposition, one can calculate
	\[
	{}^t ( z , 0 ) \cdot \nu_E ( x , t ) = z \cdot \mathbf{ p } ( \nu_{ E } ( x , t ) ) = \lv \mathbf{ p } ( \nu_{ E } ( x , t ) ) \rv ( z \cdot \nu_{ E_t } ( x ) ) = 0
	\]
	for all $ z \in T_x \, ( \partial^* E_t ) $. This completes the proof of part (5).
\end{proof}

\begin{proof}[Proof of Theorem \ref{thm:L2 flow area change}]
	We fix $ \phi \in C^1_c ( \R^{ n + 1 } \times ( 0 , \infty ) ) $ arbitrarily. Then, by the Gauss--Green theorem for sets of finite perimeter, we have
	\begin{equation}
		\label{eq: S times Gauss Green}
		\int_{ E } \partial_t \phi \, d x d t = \int_{ \partial^* E } \phi \mathbf{ q } ( \nu_{ E } ) \, d \mathcal{ H }^{ n + 1 }.
	\end{equation}
	Let $ G $ be the set of points satisfying Proposition \ref{prop: basic} (1)-(5). Then, for all $ ( x , t ) \in G $, we have
	\begin{equation}
		\label{eq: T mu}
		T_{ ( x , t ) } \, \mu = ( T_x \, ( \partial^* E_t ) \times \{ 0 \} ) \oplus \mathrm{span} \begin{pmatrix}
			v ( x , t )\\
			1
		\end{pmatrix}.
	\end{equation}
	By $ v ( x , t ) \perp T_x \, \lV V_t \rV $, \eqref{eq: T mu} and Proposition \ref{prop: basic} (1) and (4), we obtain
	\begin{equation}
		\label{eq: nu S(i) calculation}
		\nu_{ E } ( x , t ) = \frac{ 1 }{ \sqrt{ 1 + \lv v ( x , t ) \rv^2 } } \begin{pmatrix}
			\nu_{ E_t } ( x )\\
			- v ( x , t ) \cdot \nu_{ E_t } ( x )
		\end{pmatrix}.
	\end{equation}
	By \eqref{eq: nu S(i) calculation}, $ v ( x , t ) \perp \Tan_x \lV V_t \rV $ and Proposition \ref{prop: basic} again, for all $ ( x , t ) \in G $, we can calculate the $ i \times ( n + 1 ) $ component of the matrix $ I_{ n + 1 } - \nu_E \otimes \nu_E $ for $ i = 1 , \ldots , n + 1 $ as
	\[
	( I_{ n + 1 } - \nu_E \otimes \nu_E )_{ i , n + 1 } ( x , t )
	=\begin{cases}
		\frac{ - \nu_{ E_t } ( x )_i ( v ( x , t ) \cdot \nu_{ E_t } ( x ) ) }{ 1 + \lv v ( x , t ) \rv^2 } &( i = 1 , \ldots , n ),\\
		\frac{ 1 }{ 1 + \lv v ( x , t ) \rv^2 } &( i = n + 1 ),
	\end{cases}
	\]
	where $ I_{ n + 1 } $ is the ($ n + 1 $)-identity matrix and $ ( \nu_E )_i $ is the $ i $-th component of $ \nu_E $. According to this calculation and the facts the $ \nabla \mathbf{ q } = \e_{ n + 1 } $ and Proposition \ref{prop: basic} (3), we obtain that the co-area factor of the projection $ \mathbf{ q } $ satisfies
	\begin{equation}
		\label{eq: coarea factor}
		\lv \nabla^{ \partial^* E } \mathbf{ q } ( \nu_{ E } ( x , t ) ) \rv = \frac{ 1 }{ \sqrt{ 1 + \lv v ( x , t ) } \rv^2 }.
	\end{equation}
	Thanks to \eqref{eq: S times Gauss Green}-\eqref{eq: coarea factor} and the coarea formula, we compute
	\begin{equation*}
		\begin{split}
			&\int_{ E } \partial_t \phi \, d x d t = - \int_G \frac{ \phi v \cdot \nu_{ E_t } }{ \sqrt{ 1 + \lv v \rv^2 } } \, d \mathcal{ H }^{ n + 1 }
			= - \int_{ \partial^* S \cap \{ t > 0 \} } \phi v \cdot \nu_{ E_t } \lv \nabla^{ \partial^* E } \mathbf{ q } ( \nu_{ E } ) \rv \, d \mathcal{ H }^{ n + 1 }\\
			&= - \int_0^{ \infty } \int_{ \partial^* E \cap \{ \mathbf{ q } = t \} } \phi v \cdot \nu_{ E_t } \, d \mathcal{ H }^n d t
			= - \int_0^{ \infty } \int_{ \R^{ n + 1 } } \phi v \cdot \nu_{ E_t } \, d \lV \nabla \chi_{ E_t } \rV d t,
		\end{split}
	\end{equation*}
	where we used $ \mathcal{ H }^{ n + 1 } ( \partial^* E \setminus G ) = 0 $. By a suitable approximation of $ \phi $ and the continuity of the $ L^2 $ norm of $ \chi_E $ (assumption \eqref{itm:continuity assumption} in Theorem \ref{thm:L2 flow area change}), for all $ 0 \leq t_1 < t_2 $, we deduce \eqref{eq:v area change formula}.
\end{proof}

\section{Proof of Theorem \ref{thm:main result}}
\label{sec: proof of main}
To prove Theorem \ref{thm:main result}, it suffices to show that the flow constructed in \cite{liu2024existence} is an $ L^2 $ flow with velocity $ v = h + u^\perp $ satisfying the hypothesis of Theorem \ref{thm:L2 flow area change}. We fix a given vector field $ u \in L^{ \infty }_{ loc } ( [ 0 , \infty ) ; L^2 ( \R^2 ) ) \cap L^2_{ loc } ( [ 0 , \infty ) ; W^{ 1 , 2 } ( \R^2 ) ) $. We let $ \Gamma_0 $ and $ E_{ 0 , 1 } , \ldots , E_{ 0 , N } $ as in Assumption \ref{assum: initial} and $ V^{ ( m ) }_t $, $ E_i^{ ( m ) } ( t ) $, $ u^{ ( m ) } $, $ V_t $ and $ E_i ( t ) $ as in Subsection \ref{subsec: review}, and let $ h^{ ( m ) } = h^{ ( m ) } ( \cdot , V_t^{ ( m ) } ) $ be the generalized mean curvature of $ V_t^{ ( m ) } $. Furthermore, we let $ d \mu := d \lV V_t \rV d t $ and $ d \mu^{ ( m ) } := d \lV V_t^{ ( m ) } \rV d t $.

When taking a limit, by using the compactness of set of finite perimeter, one obtains the following (see \cite[Proposition 5.7]{liu2024existence}).

\begin{prop}
\label{prop: S convergence}
	Let $ S^{ ( m ) } ( i ) $ be as in Theorem \ref{thm: existence for smooth u} (7) corresponding to each $ m $. Taking a subsequence if necessary, we have the following:
	\begin{enumerate}
		\item For all $ t \geq 0 $ and for each $ i = 1 , \ldots N $, $ \chi_{ E_i^{ ( m ) } ( t ) } \to \chi_{ E_i ( t ) } $ locally in $ L^1 ( \R^2 ) $ and $ E_i ( t ) $ is a set of finite perimeter.
		\item For each $ i = 1 , \ldots , N $, $ \chi_{ S^{ ( m ) } ( i ) } \to \chi_{ S ( i ) } $ locally in $ L^1 ( \R^2 \times [ 0 , \infty ) ) $ and $ S ( i ) $ is a set of finite perimeter.
	\end{enumerate}
\end{prop}

We prove the absolute continuity assumption \eqref{itm:measure assumption} in Theorem \ref{thm:L2 flow area change} in the following lemma.

\begin{lemma}
\label{lem: S<<V}
	For each $ i = 1 , \ldots , N $, we have $ \mathcal{ H }^2 \lfloor_{ \partial^* S ( i ) } \ll \mu $ in $ \R^2 \times ( 0 , \infty ) $. 
\end{lemma}
\begin{proof}
	Fix $ i \in \{ 1 , \ldots , N \} $. For each $ m $, one has
	\[
		\int_{ S^{ ( m ) } ( i ) } \partial_t \phi \, d x d t = - \int_0^{ \infty } \int_{ \R^2 } ( h^{ ( m ) } + u^{ ( m ) } ) \cdot \nu_{ E_i^{ ( m ) } ( t ) } \, d \lV \nabla \chi_{ E_i^{ ( m ) } ( t ) } \rV d t
	\]
	for all $ \phi \in C^1_c ( \R^2 \times ( 0 , \infty ) ) $ by \eqref{eq: BV u is smooth}. Since $ V_t^{ ( m ) } $ satisfies the a priori estimate \eqref{eq: mass a priori} and Theorem \ref{thm: existence for smooth u} (6), the family of space-time Radon measures $ \{ d \lV \nabla \chi_{ E_i^{ ( m ) } ( t ) } \rV d t \}_{ m \in \N } $ is uniformly local bounded. Thus, by passing to a further subsequence if necessary, there exists a Radon measure $ d \alpha_i $ in $ \R^2 \times ( 0 , \infty ) $ such that $ d \lV \nabla \chi_{ E_i^{ ( m ) } ( t ) } \rV d t \rightharpoonup d \alpha_i $ as $ m \to \infty $. By Theorem \ref{thm: existence for smooth u} (6) again, $ \alpha_i $ satisfies, for all $ \phi \in C_c ( \R^2 \times ( 0 , \infty ) ) $ with $ \phi \geq 0 $,
	\[
		\alpha_i ( \phi ) = \lim_{ m \to \infty } \int_0^{ \infty } \int_{ \R^2 } \phi \, d \lV \nabla \chi_{ E_i^{ ( m ) } ( t ) } \rV d t \leq \lim_{ m \to \infty } \mu^{ ( m ) } ( \phi ) = \mu ( \phi ).
	\]
	In particular, $ \alpha_i \ll \mu $. We let the Radon--Nykod\'{y}m derivative of $ \alpha_i $ with respect to $ d \lV V_t \rV d t $ be $ f_i := d \alpha_i / d \mu $. Note that $ \lv f_i \rv \leq 1 $ follows from the above inequality $ \mu $-almost everywhere.
	
	Next, let us consider the integrand of the following
	\[
		\int_0^{ \infty } \int_{ \R^2 } ( h^{ ( m ) } + u^{ ( m ) } ) \cdot \nu_{ E_i^{ ( m ) } ( t ) } \, d \lV \nabla \chi_{ E_i^{ ( m ) } ( t ) } \rV d t.
	\]
	By Theorem \ref{thm: existence for smooth u} (6), \eqref{eq: h a priori} and \eqref{eq: u a priori}, we obtain the uniform bound
	\[
		\int_0^{ \infty } \int_{ \R^2 } \lv ( h^{ ( m ) } + u^{ ( m ) } ) \cdot \nu_i^{ ( m ) } \rv^2 \, d \lV \nabla \chi_{ E_i^{ ( m ) } ( t ) } \rV d t \leq 2 \int_0^{ \infty } \int_{ \R^2 } ( \lv h^{ ( m ) } \rv^2 + \lv u^{ ( m ) } \rv^2 ) \, d \lV V_t^{ ( m ) } \rV d t < \infty.
	\]
	By applying Theorem \ref{thm: measure function cptness} to the pairs $ ( (  h^{ ( m ) } + u^{ ( m ) } ) \cdot \nu_i^{ ( m ) } , d \lV \nabla \chi_{ E_i^{ ( m ) } ( t ) } \rV d t ) $, taking a subsequence if necessary, there exists a function $ v_i \in L^2_{ loc } ( \alpha_i ) $ such that, for all $ \phi \in C_c ( \R^2 \times ( 0 , \infty ) ) $, 
	\[
		\lim_{ m \to \infty } \int_0^{ \infty } \int_{ \R^2 } \phi ( h^{ ( m ) } + u^{ ( m ) } ) \cdot \nu_{ E_i^{ ( m ) } ( t ) } \, d \lV \nabla \chi_{ E_i^{ ( m ) } ( t ) } \rV d t = \int_0^{ \infty } \int_{ \R^2 } \phi v_i \, d \alpha_i.
	\]
	Combining with the above arguments and Proposition \ref{prop: S convergence}, for all $ \phi \in C^1_c ( \R^2 \times ( 0 , \infty ) ) $, we obtain
	\begin{equation*}
		\begin{split}
			\int_{ S ( i ) } \partial_t \phi \, d x d t
			&= \lim_{ m \to \infty } \int_{ S^{ ( m ) } ( i ) } \partial_t \phi \, d x d t
			= \lim_{ m \to \infty } - \int_0^{ \infty } \int_{ \R^2 } \phi ( h^{ ( m ) } + u^{ ( m ) } ) \cdot \nu_i^{ ( m ) } \, d \lV \nabla \chi_{ E_i^{ ( m ) } ( t ) } \rV d t\\
			&= - \int_0^{ \infty } \int_{ \R^2 } \phi v_i \, d \alpha_i
			= - \int_0^{ \infty } \int_{ \R^2 } \phi v_i f_i \, d \lV V_t \rV d t.
		\end{split}
	\end{equation*}
	It follows from this that $ d \nabla' \chi_{ S ( i ) } = d \nabla \chi_{ E_i ( t ) } d t , v_i f_i \, d \lV V_t \rV d t $ in the sense of vactorial Radon measures, where $ d \nabla' \chi_{ S ( i ) } $ is the space-time distributional derivative of $ \chi_{ S ( i ) } $. By Theorem \ref{thm: existence for smooth u} (6) again and noting that $ \lv f_i \rv \leq 1 $ $ \mu $-almost everywhere, we have $ \mathcal{ H }^2 \lfloor_{ \partial^* S ( i ) } \ll \mu $.
\end{proof}

By Proposition \ref{prop: a priori}, Lemma \ref{lem: S<<V}, Theorem \ref{thm: existence for critical u} (8) and Proposition \ref{prop: S convergence}, the flow satisfies assumptions (1)--(3) of Theorem \ref{thm:L2 flow area change}. It thus remains only to prove that the flow is an $ L^2 $ flow with velocity $ v = h + u^\perp $. This fact follows from a slight modification of the argument in \cite[Theorem 4.3]{StuvardTonegawa+2022}; nevertheless, for completeness, we provide the proof below.

\begin{prop}
\label{prop:L2}
	Let $ \{ V_t \}_{ t \geq 0 } $ be as in Theorem 2.4. Then, for any $ 0 < T < \infty $, $ \{ V_t \}_{ t \in [ 0 , T ) } $ is an $ L^2 $ flow with velocity $ v = h + u^\perp $ in $ \R^2 \times [ 0 , T ) $ defined as in Definition \ref{def:L2 flow}. 
\end{prop}
	\begin{proof}
	We confirm the requirements of Definition \ref{def:L2 flow} Since $ \{ V_t \}_{ t \geq 0 } $ is as in Theorem \ref{thm: existence for smooth u}, the conditions (a), (b) and (c'1) are satisfied. To complete the proof, we are only left with checking that \eqref{eq:L2 flow} holds.

	Let $ 0 < T < \infty $, $ \phi \in C^1_c ( \R^2 \times ( 0 , T ) ) $ be fixed. First, we consider that a test function is $ \phi \geq 0 $ and $ \phi \neq 0 $. From \eqref{eq: Brakke}, we deduce
	\begin{equation*}
		\begin{split}
			\int_0^T \int_{ \R^2 } \partial_t \psi + \nabla \psi \cdot \left( h + u^{ \perp } \right) - \psi\, h \cdot u^{ \perp } \, d \lV V_t \rV d t
			\geq \int_0^T \int_{ \R^2 } \psi \lv h \rv^2 \, d \lV V_t \rV d t \geq 0
		\end{split}
	\end{equation*}
	for all $ \psi \in C^1_c ( \R^2 \times [ 0 , T ] ) $ with $ \psi \geq 0 $. We here define the linear functional $ L $ on $ C^1_c ( \R^2 \times [ 0 , T ] ; [ 0 , \infty ) ) $ by
	\[
		L \psi := \int_0^T \int_{ \R^2 } \partial_t \psi + \nabla \psi \cdot \left( h + u^{ \perp } \right) - \psi\,h \cdot u^{ \perp } \, d \lV V_t \rV d t,
	\]
	then $ L $ is monotone, that is, $ L\,\psi_1 \leq L\,\psi_2 $ whenever $ \psi_1 \leq \psi_2 $ everywhere. For every small number $ \varepsilon > 0 $, we let take a cutoff function $ \Phi_{ \delta } ( t ) $ defined by the following conditions:
	\begin{description}
		\item[(i)] $ \Phi_{ \delta } $ is a $ C^1_c $ function on $ t \in \R^+ $ and $ 0 \leq \Phi_{ \delta } \leq 1 $,
		\item[(ii)] $ \Phi_{ \delta } \equiv 1 $ on $ ( 2 \delta , T ] $ and $ \Phi_{ \delta } \equiv 0 $ on $ [ 0 , \delta ] $,
		\item[(iii)] there exists a constant $ C > 0 $ such that  $ \sup_{ \R^+ } \lv \partial_t \Phi_{ \delta } \rv \leq C / \delta $.
	\end{description}
	We take a sufficiently small number $ \delta > 0 $ so that $ \mathrm{ supp }\,\phi \subset \R^2 \times [ 2 \varepsilon , T ] $. Then the definition of $ \Phi_{ \delta } $ implies
	\begin{equation}
		0 \leq \frac{ \phi }{ \| \phi \|_{ C^0 } } \leq \Phi_{ \delta }.
		\label{monoto}
	\end{equation}
	On the other hand, we compute
	\begin{align*}
		&\lv L \Phi_{ \delta } \rv
		\leq \int_0^T \int_{ \R^2 } \lv \partial_t \Phi_{ \delta } \rv + \Phi_{ \delta } \lv h \rv \lv u \rv \, d \lV V_t \rV d t\\
		&\leq \int_{ \varepsilon }^{ 2 \varepsilon } \int_{ \R^2 } \frac{ C }{ \varepsilon } \, d \lV V_t \rV d t + \| h \|_{ L^2 ( \mu , \R^2 \times [ 0 , T ] ) } + \| u \|_{ L^2 ( \mu , \R^2 \times [ 0 , T ] ) }\\
		&\leq C \sup_{ 0 \leq t \leq T } \lV V_t \rV ( \R^2 ) + \| h \|_{ L^2 ( \mu , \R^2 \times [ 0 , T ] ) } + \| u \|_{ L^2 ( \mu , \R^2 \times [ 0 , T ] ) },
	\end{align*}
	where we use the H\"{o}lder inequality. This inequality, \eqref{monoto} and monotonicity of $ L $ imply that
	\begin{align*}
		&\frac{ 1 }{ \| \phi \|_{ C^0 } } \left| \int_0^T \int_{ \R^2 } \partial_t \phi + \nabla \phi \cdot \left( h + u^{ \perp } \right) d \lV V_t \rV d t\,\right|\\
		&\leq \frac{ | L \phi\,| }{ \| \phi \|_{ C^0 } } + \frac{ 1 }{ \| \phi \|_{ C^0 } } \left| \int_0^T \int_{ \R^2 } \phi\,h \cdot u^{ \perp } \, d \lV V_t \rV d t\,\right|\\
		&\leq | L\,\Phi_{ \delta }\,| + \| h \|_{ L^2 ( \mu , \R^2 \times [ 0 , T ] ) } + \| u \|_{ L^2 ( \mu , \R^2 \times [ 0 , T ] ) }
		\leq C ( \mu , \R^2, T ),
	\end{align*}
	where we used the H\"{o}lder inequality again and $ C ( \mu , T ) $ depends only on $ \mu $ and $ T $. Thus we obtain \eqref{eq:L2 flow} for any positive test functions. If the test function $ \phi $ is general case, \eqref{eq:L2 flow} is proved by splitting $ \phi = \phi_+ - \phi_- $, where $ \phi_+ = \max( \phi , 0 ) $ and $ \phi_- = \max ( - \phi , 0 ) $, and arguing the same as above for each $ \phi_+ $ and $ \phi_- $, respectively.
\end{proof}

\section{Compactness for generalized BV flows}
\label{sec:compactness}
Although the compactness theorem for Brakke flows is already known, no corresponding result had been established for generalized BV flows, which also keeps track of the grains. In this section, we discuss a compactness property for generalized BV flows by combining Theorem \ref{thm:L2 flow area change} with the compactness for sets of finite perimeter. The proof is essentially contained in \cite{StuvardTonegawa+2022,liu2024existence}. Nevertheless, we provide an explicit outline of the argument here.

We first introduce the definition of the Brakke flow and the generalized BV flow. The Brakke flow, first introduced by Brakke \cite{brakke1978motion}, is a family of varifolds satisfying a localized surface energy dissipation as a weak notion of solution to MCF. In the co-dimension $ 1 $ case, Stuvard--Tonegawa \cite{StuvardTonegawa+2022} observed that considering a Brakke flow satisfying \eqref{eq: BV_intro} provides essential advantage over the Brakke flow. This led them to introduce the generalized BV flow. 

\begin{definition}[Brakke flow]
	\label{def:Brakke}
	A family of varifolds $ \{ V_t \}_{ t \geq 0 } \subset \V_n ( \R^{ n + 1 } ) $ is \textit{an} $ n $-\textit{dimensional Brakke flow} if the following three conditions are satisfied:
	\begin{description}
		\item[\textup{(1)}] For almost every $ t \in [ 0 , \infty ) $, $ V_t \in \IV_n ( \R^{ n + 1 } ) $ and 
		$ \delta V_t $ is locally bounded and absolutely continuous with respect to $ \lV V_t \rV $ (thus the generalized mean curvature exists for almost every $ t $, denoted by $ h $).
		\item[\textup{(2)}] For all $ T > 0 $ and all compact set $ K \subset \R^{ n + 1 } $
		\[
		\sup_{ t \in [ 0 , T ] } \lV V_t \rV ( K ) < \infty, \quad \int_0^T \int_K \lv h \rv^2 \, d \lV V_t \rV d t < \infty.
		\]
		\item[\textup{(3)}] For all $ 0 \leq t_1 < t_2 < \infty $ and all test function $ \phi \in C^1_c ( \R^{ n + 1 } \times [ 0 , \infty ) ; [ 0 , \infty ) )$,
		\begin{equation}
			\label{eq:varifold Brakke ineq}
			\begin{split}
				&\lV V_{ t_2 } \rV ( \phi ( \cdot , t_2 ) ) - \lV V_{ t_1 } \rV ( \phi ( \cdot , t_1 ) )\\
				&\leq \int_{ t_1 }^{ t_2 } \int_{ \R^{ n + 1 } } \Big( \nabla \phi ( x , t ) - \phi ( x , t )\,h ( x , t ) \Big) \cdot h ( x , t ) + \partial_t \phi ( x , t ) \, d \lV V_t \rV ( x ) dt.
			\end{split}
		\end{equation}
	\end{description}
\end{definition}

\begin{definition}[Generalized BV flow]
	\label{def:genaralized BV}
	Let $ \{ V_t \}_{ t \geq 0 } $ and $ \{ E_t \}_{ t \geq 0 } $ be families of varifolds and sets of finite perimeter, respectively. The pair $ ( \{ V_t \}_{ t \geq 0 } , \{ E_t \}_{ t \geq 0 } ) $ is \textit{a generalized BV flow} if all of the following hold:
	\begin{description}
		\item[\quad\textrm{(i)}] $ \{ V_t\}_{ t \geq 0 } $ is an $ n $-dimensional Brakke flow.
		\item[\quad\textrm{(ii)}] For all $ t \geq 0 $, $ \lV \nabla \chi_{ E_t } \rV \leq \lV V_t \rV $.
		\item[\quad\textrm{(iii)}] For all $ 0 \leq t_1 < t_2 < \infty $ and all test function $ \phi \in C^1_c ( \R^{ n + 1 } \times [ 0 , \infty ) ) $,
		\begin{equation}
			\label{eq:generalized BV}
			\begin{split}
				\int_{ E_t } \phi \, d x \bigg\rv^{ t_2 }_{ t = t_1 } = \int_{ t_1 }^{ t_2 } \int_{ E_t } \partial_t \phi \, d x d t
				+ \int_{ t_1 }^{ t_2 } \int_{ \partial^* E_t } \phi ( h \cdot \nu_{ E_t } ) \, d \mathcal{ H }^n d t,
			\end{split}
		\end{equation}
		where $ h $ is the generalized mean curvature of $ V_t $.
	\end{description}
\end{definition}

The compactness statement for generalized BV flows is as follows:

\begin{theorem}
	\label{thm:compactness}
	Let $ ( \{ V_t^i \}_{ t \geq 0 } , \{ E_t^i \}_{ t \geq 0 } ) $ be a family of generalized BV flows as in Definition \ref{def:genaralized BV}. Additionally, we assume that, for all $ 0 < T < \infty $ and all compact set $ K \subset \R^{ n + 1 } $,
	\begin{equation}
		\label{eq:compactness assumption 1}
		\sup_i \sup_{ t \in [ 0 , T ] } \lV V_t^i \rV ( K ) \leq C ( T , K ).
	\end{equation}
	Then, there exist a subsequence $ \{ i_j \} \subset \{ i \} $, a family of varifolds and sets of finite perimeter $ \{ V_t \}_{ t \geq 0 } $ and $ \{ E_t \}_{ t \geq 0 } $ such that the pair $ ( \{ V_t \}_{ t \geq 0 } , \{ E_t \}_{ t \geq 0 } ) $ is a generalized BV flow and, for all $ t \geq 0 $,
	\[
	\lV V_t^{ i_j } \rV \to \lV V_t \rV \textup{ as Radon measures and } E^{ i_j } \to E \textup{ in $ L^1_{loc} ( \R^{ n + 1 } \times [ 0 , \infty ) ) $,}
	\]
	where $ E^i = \{ ( x , t ) \in \R^{ n + 1 } \times [ 0 , \infty ) \mid x \in E^i_t \} $ and $ E = \{ ( x , t ) \in \R^{ n + 1 } \times [ 0 , \infty ) \mid x \in E_t \} $. Moreover, for every $ t \in [ 0 , \infty ) $, there exists a subsequence $ \{ i'_j \} \subset \{ i_j \} $ which may depend on $ t $ such that
	\[
	V^{ i'_j } \to V_t \textup{ as varifolds and } E^{ i'_j }_t \to E_t \textup{ in $ L^1_{ loc } ( \R^{ n + 1 } ) $.}
	\]
\end{theorem}

\begin{remark}
	Brakke flow theory yields the a priori curvature estimate
	\begin{equation}
		\label{eq:compactness assumption 2}
		\int_0^T \int_K \phi \lv h \rv^2 \, d \lV V_t \rV d t \leq \lV V_0 \rV ( \phi ) + \int_0^T \int \frac{ \lv \nabla \phi \rv^2 }{ \phi } \, d \lV V_t \rV d t
	\end{equation}
	for all $ \phi \in C^2_c ( \R^{ n + 1 } ; [ 0 , \infty ) ) $. See \cite[Lemma 3.2]{tonegawa2019brakke} for details. Consequently, in the compactness argument, the curvature term is uniformly controlled independently of the index $ i $, provided the mass of $ V_t^i $ is uniformly bounded.
\end{remark}

Let $ ( \{ V_t^i \}_{ t \geq 0 } , \{ E_t^i \}_{ t \geq 0 } ) $ satisfy the hypotheses of Theorem \ref{thm:compactness}. Applying the compactness theorem for Brakke flows (see \cite[Theorem 3.7]{tonegawa2019brakke}) to $ \{ V_t^i \}_{ t \geq 0 } $, we obtain a limit Brakke flow $ \{ V_t \}_{ t \geq 0 } $ that Theorem \ref{thm:compactness} requires. The main issue that remain are therefore the convergence of $ \{ E^i_t \}_{ t \geq 0 } $ and whether \eqref{eq:generalized BV} is preserved in the limit. We first use the compactness theorem for sets of finite perimeter to identify a limit of $ \{ E_t^i \}_{ t \geq 0 } $.

\begin{prop}
	\label{prop:E convergence}
	Let $ E^i = \{ ( x , t ) \in \R^{ n + 1 } \times [ 0 , \infty ) \mid x \in E^i_t \} $. For the sequence $ \{ E^i \}_{ i = 1 }^\infty $, there exist a subsequence of $ \{ i_j \} \subset \{ i \} $ and a set of finite perimeter $ E $ such that
	\[
	E^{ i_j } \to E \textup{ in $ L^1_{ loc } ( \R^{ n + 1 } \times [ 0 , \infty ) ) $. }
	\]
	Moreover, for every $ t \in [ 0 , \infty ) $, there exists a subsequence $ \{ i'_j \} \subset \{ i_j \} $ which may depend on $ t $ such that
	\[
	E^{ i'_j }_t \to E_t \textup{ in $ L^1_{ loc } ( \R^{ n + 1 } ) $ },
	\]
	where $ E_t = \{ x \in \R^{ n + 1 } \mid ( x , t ) \in E \} $, and we have $ \lV \nabla \chi_{ E_t } \rV \leq \lV V_t \rV $.
\end{prop}
\begin{proof}
	For any test vector field $ X = ( X_1 , \ldots , X_{ n + 1 } , X_t ) \in C^1_c ( \R^{ n + 1 } \times \R ; \R^{ n + 2 } ) $, from \eqref{eq:generalized BV}--\eqref{eq:compactness assumption 2} and Definition \ref{def:genaralized BV} (ii), it follows that
	\begin{equation*}
		\begin{split}
			&\bigg\lv \int \int_{ E^i } \dv X \, d x d t \bigg\rv
			= \bigg\lv \int_0^\infty \int_{ E_t^i } \nabla \cdot ( X_1 , \ldots , X_{ n + 1 } ) \, d x d t + \int_0^\infty \int_{ E_t^i } \partial_t X_t \, d x d t \bigg\rv\\
			&= \bigg\lv \int_0^\infty \int_{ \partial^* E_t^i } ( X_1 , \ldots , X_{ n + 1 } ) \cdot \nu_{ E_t^i } \, d \mathcal{ H }^n d t - \int_{ E_0^i } X_t ( x , 0 ) \, d x - \int_0^\infty \int_{ \partial^* E_t^i } X_t ( h^i \cdot \nu_{ E_t^i } ) \, d \mathcal{ H }^n d t \bigg\rv\\
			&\leq \sup_{ \R^{ n + 1 } \times \R } \lv X \rv \bigg( T \sup_{ t \in [ 0 , T ) } \lV \nabla \chi_{ E_t^i } \rV ( B_R ) + \mathcal{ L }^{ n + 1 } ( E_0^i \cap B_R ) + \int_0^T \int_{ B_R } \lv h^i \rv \, d \lV \nabla \chi_{ E_t^i } \rV d t \bigg)\\
			&\leq \sup_{ \R^{ n + 1 } \times \R } \lv X \rv \bigg( T \sup_{ t \in [ 0 , T ) } \lV V_t^i \rV ( B_R ) + \mathcal{ L }^{ n + 1 } ( B_R )\\
			&\quad \quad \quad + \Big( T \sup_{ t \in [ 0 , T ) } \lV V_t \rV ( B_R ) \Big)^{ 1 / 2 } \bigg( \int_0^T \int_{ B_R } \lv h^i \rv^2 \, d \lV V_t^i \rV d t \bigg)^{ 1 / 2 } \bigg) \leq C \sup_{ \R^{ n + 1 } \times \R } \lv X \rv,
		\end{split}
	\end{equation*}
	where we chose $ T > 0 $ and $ R > 0 $ so that $ \spt \lv X \rv \subset B_R \times ( - T , T ) $. Thus $ E^i $ is a set of locally finite perimeter and the upper bound of the perimeter measures $ \lV \nabla' \chi_{ E^i } \rV $ is independent of $ i $. By the compactness for sets of finite perimeter, taking a subsequence $ \{ i_j \} \subset \{ i \} $ if necessary, there is a set of finite perimeter $ E \subset \R^{ n + 1 } \times [ 0 , \infty ) $ such that $ E^{ i_j } \to E $ in $ L^1_{ loc } $.
	
	It is obvious that $ \chi_E \in BV_{ loc } ( \R^{ n + 1 } \times [ 0 , \infty ) ) $. Note that, by Fubini's Theorem, one can see $ \chi_{ E_t^i } \to \chi_{ E_t } $ in the $ L^1_{ loc } $ sense for almost every $ t \in ( 0 , \infty ) $. Choose $ t \in [ 0 , \infty ) $ such that this convergence does not hold. By Definition \ref{def:genaralized BV} (ii) and \eqref{eq:compactness assumption 1}, we see that
	\[
	\lV \nabla \chi_{ E_t^{ i_j } } \rV ( \phi ) \leq \lV V_t^{ i_j } \rV ( \phi ) \leq C ( \spt \phi ) \sup \lv \phi \rv
	\]
	for given $ t $ and for all $ \phi \in C^\infty_c ( \R^{ n + 1 } ; [ 0 , \infty ) ) $. By the compactness for sets of finite perimeter again, taking a subsequence $ \{ i'_j \} \subset \{ i_j \} $ if necessary, there is a set of finite perimeter $ E_t \subset \R^{ n + 1 } $ such that $ E_t^{ i'_j } \to E_t $ in $ L^1_{ loc } $. We may re-define $ E $ such that every time slice of $ E $ becomes the limit $ E_t $ while preserving the convergence $ E^{ i_j } \to E $ established above. The lower semi-continuity of the perimeter measure and Definition \ref{def:genaralized BV} (ii), $ \lV \nabla \chi_{ E_t } \rV \leq \lV V_t \rV $ holds. This proves the claim.
\end{proof}

It remains to prove \eqref{eq:generalized BV} for the pair $ ( \{ V_t \}_{ t \geq 0 } , \{ E_t \}_{ t \geq 0 } ) $. To this end, we verify that the pair $ ( \{ V_t \}_{ t \geq 0 } , \{ E_t \}_{ t \geq 0 } ) $ satisfies the assumptions of Theorem \ref{thm:L2 flow area change}. We note that, for a usual Brakke flow, Huisken's monotonicity formula gives the density bound assumption \eqref{itm:density assumption} (see \cite[Proposition 3.5]{tonegawa2019brakke}, for example). Moreover, by arguments entirely parallel to Lemma \ref{lem: S<<V} and Proposition \ref{prop:L2}, we obtain the following:

\begin{prop}
	For the pair $ ( \{ V_t \}_{ t \geq 0 } , \{ E_t \}_{ t \geq 0 } ) $ discussed above, we have
	\begin{enumerate}
		\item $ d \mathcal{ H }^n \lfloor_{ \partial^* E } \ll d \lV V_t \rV d t $ in $ \R^{ n + 1 } \times ( 0 , \infty ) $; and
		\item $ \{ V_t \}_{ t \geq 0 } $ is an $ L^2 $ flow with velocity $ v = h $.
	\end{enumerate}
\end{prop}

Finally, once the H\"{o}lder continuity of $ \chi_E $ holds, the compactness follows from Theorem \ref{thm:L2 flow area change}.

\begin{prop}
	\label{prop:Holder continuity of E}
	Let $ E $ be as in Proposition \ref{prop:E convergence}, then
	\[
	\chi_E \in C^{ 1 / 2 }_{ loc } ( [ 0 , \infty ) ; L^1_{ loc } ( \R^{ n + 1 } ) ) \cap BV_{ loc } ( \R^{ n + 1 } \times [ 0 , \infty ) ).
	\]
\end{prop}
\begin{proof}
	Let $ \phi \in C_c^\infty ( \R^{ n + 1 } ) $ and we fix $ 0 \leq t_1 < t_2 $ arbitrarily. For these $ t_1 , t_2 $, from Proposition \ref{prop:E convergence}, it follows that $ E^{ i'_j }_{ t_1 } \to E_{ t_1 } $ and $ E^{ i'_j }_{ t_2 } \to E_{ t_2 } $ by taking a further subsequence if necessary. Choosing a cut-off function $ \eta \in C_c^\infty ( [ 0 , \infty ) ) $ such that $ \eta \equiv 1 $ on $ [ t_1 , t_2 ] $ implies
	\begin{equation*}
		\begin{split}
			&\bigg\lv \int_{ E_{ t_2 }^{ i'_j } } \phi ( x ) \, d x - \int_{ E_{ t_1 }^{ i'_j } } \phi ( x ) \, d x \bigg\rv
			= \bigg\lv \int_{ t_1 }^{ t_2 } \int_{ \partial^* E_t^{ i'_j } } \phi ( h^i \cdot \nu_{ E_t^{ i'_j } } ) \, d \mathcal{ H }^n d t \bigg\rv\\
			&\leq C ( t_2 - t_1 )^{ 1 / 2 } \bigg( \int_{ t_1 }^{ t_2 } \int \phi^2 \lv h^{ i'_j } \rv^2 \, d \lV V_t^{ i'_j } \rV d t \bigg)^{ 1 / 2 }
			\leq C ( t_2 - t_1 )^{ 1 / 2 },
		\end{split}
	\end{equation*}
	where $ C > 0 $ is independent of $ i'_j $ and we used \eqref{eq:generalized BV}--\eqref{eq:compactness assumption 2} and Definition \ref{def:genaralized BV} (ii). For any bounded open set $ U $, we may let $ A = U \cap E_{ t_2 }^{ i'_j } \setminus E_{ t_1 }^{ i'_j } $ or $ U \cap E_{ t_1 }^{ i'_j } \setminus E_{ t_2 }^{ i'_j } $ and approximate $ \chi_A $ by $ \phi $ to see that
	\[
	\mathcal{ L }^{ n + 1 } ( U \cap ( E_{ t_2 }^{ i'_j } \Delta E_{ t_1 }^{ i'_j } ) ) \leq C ( t_2 - t_1 )^{ 1 / 2 }.
	\]
	Letting $ i'_j \to \infty $, one can conclude that $ \chi_{ E_t } $ is $ 1 / 2 $-H\"{o}lder continuous in $ L^1_{ loc } $.
\end{proof}

\section{Proof of Theorem \ref{thm: main2}}
\label{sec: proof of main2}

\begin{proof}[Proof of Theorem \ref{thm: main2}]
	 We may assume that $ T^* \leq 2 \lv E ( 0 ) \rv / \lV V_0 \rV ( \R^2 )^2 =: T $, for example. Noting that one has the a priori estimates (Proposition \ref{prop: a priori}), one can test the formula \eqref{eq: BV_intro} with $ \phi \equiv 1 $.  Set $ v ( t ) := \lv E_i ( t ) \rv $. Since $ \nu_{ E_i ( t ) } \cdot h = - \nu_{ E_j ( t ) } \cdot h $ on $ \partial^* E_i ( t ) \cap \partial^* E_j ( t ) $ for $ i \neq j $ $ \mathcal{ H }^1 $-almost everywhere, summing over $ i = 1 , \ldots N - 1 $, \eqref{eq: BV_intro} implies
	 \[
	 	v^{ \prime } ( t ) = \int_{ \R^2 } ( h + u ) \cdot \nu_{ E ( t ) } \, d \lV \nabla \chi_{ E_i ( t ) } \rV \quad \text{for } a.e.\, t \geq 0.
	 \]
	 Set $ a ( t ) = \lV V_t \rV ( \R^2 ) $. By Brakke's inequality \eqref{eq: Brakke} with $ \phi \equiv 1 $, the upper derivative $ a_+^{ \prime } ( t ) $ of $ a ( t ) $ satisfies, for almost every $ t \geq 0 $,
	 \[
	 	- a_+^{ \prime } ( t ) \geq \int_{ \R^2 } h \cdot ( h + u^{ \perp } ) \, d \lV V_t \rV
	 \]
	 where we used the Cauchy--Schwarz inequality and $ \lV \nabla \chi_{ E_i ( t ) } \rV \leq \lV V_t \rV $ for all $ t \geq 0 $ and $ i $. Combining the above implies
	 \begin{equation*}
	 	\begin{split}
	 		&- v^{ \prime } ( t )
	 		\leq \big( \mathcal{ H }^1 ( \partial^* E ( t ) ) \big)^{ 1 / 2 } \bigg( \int_{ \R^2 } \lv h \rv^2 \, d \lV \nabla \chi_{ E ( t ) } \rV \bigg)^{ 1 / 2 }\\
	 		&\leq \big( \lV V_t \rV ( \R^2 ) \big)^{ 1 / 2 } \bigg( - a_+^{ \prime } ( t ) - \int_{ \R^2 } h \cdot u^{ \perp } \, d \lV V_t \rV \bigg)^{ 1 / 2 }
	 		\leq \bigg( - \frac{ ( a ( t )^2 )_+^{ \prime } }{ 2 } + a ( t ) \int_{ \R^2 } \lv h \cdot u \rv \, d \lV V_t \rV \bigg)^{ 1 / 2 }
	 	\end{split}
	 \end{equation*}
	 for almost every $ t \geq 0 $. Integrating over $ [ 0 , T ] $ and the H\"{o}lder inequality give
	 \begin{equation*}
	 	\begin{split}
	 		\lv E ( 0 ) \rv - \lv E ( T ) \rv
	 		&\leq \int_0^T \bigg( - \frac{ ( a ( t )^2 )_+^{ \prime } }{ 2 } + a ( t ) \int_{ \R^2 } \lv h \cdot u \rv \, d \lV V_t \rV \bigg)^{ 1 / 2 } d t\\
	 		&\leq \sqrt{ \frac{ T }{ 2 } } \bigg( \lV V_0 \rV ( \R^2 )^2 - \lV V_T \rV ( \R^2 )^2 + 2 \sup_{ 0 \leq t \leq T } \lV V_t \rV ( \R^2 ) \int_0^T \int_{ \R^2 } \lv h \cdot u \rv \, d \lV V_t \rV d t \bigg)^{ 1 / 2 }
	 	\end{split}
	 \end{equation*}
	 Since $ \lV V_t \rV ( \R^2 ) \neq 0 $ as long as $ E ( t ) \neq \emptyset $, the extinction time $ T^* $ is at least equal to the first time when $ E ( t ) = \emptyset $. Therefore, the above inequality implies that
	 \[
	 	T^* \geq \frac{ 2 \lv E ( 0 ) \rv^2 }{ \lV V_0 \rV ( \R^2 )^2 + 2 \sup_{ 0 \leq t \leq T } \lV V_t \rV ( \R^2 ) \int_0^T \int_{ \R^2 } \lv h \cdot u \rv \, d \lV V_t \rV d t } \geq \frac{ 2 \lv E ( 0 ) \rv^2 }{ \lV V_0 \rV ( \R^2 )^2 ( 1 + C ( u , T ) ) },
	 \]
	 where $ C ( u , T ) $ arises from Proposition \ref{prop: a priori}.
\end{proof}

\begin{remark}
	We assume that $ u $ is divergence-free, that is, $ \dv u = 0 $, $ N = 2 $ and every connected component of $ E_1 ( t ) $ is a simply connected bounded open set for almost every $ t \geq 0 $. Then, one can formally obtain the upper bound of the extinction time as follows: When $ N = 2 $, Liu--Tonegawa showed that there exists a finite number of connected embedded $ W^{ 2 , 2 } $ curves such that $ V_t $ is an associated varifold of the $ W^{ 2 , 2 } $ curves and each curve may intersect other curves only tangentially (\cite[Theorem 2.2 (V3)]{liu2024existence}). Thus, by a suitable approximation and the Gauss--Bonnet theorem, one may have
	\begin{equation*}
		\begin{split}
			\lv E_1 ( t ) \rv - \lv E_1 ( 0 ) \rv
			&= \int_0^t \int_{ \R^2 } ( h + u ) \cdot \nu_{ E_1 ( s ) } \, d \lV \nabla \chi_{ E_1 ( s ) } \rV d s \leq - 2 \pi t + \int_0^t \int_{ E_1 ( s ) } \dv u \, d x d t = - 2 \pi t
		\end{split}
	\end{equation*}
	for $ 0 < t < \infty $, where we used the Gauss--Green theorem. Therefore, one may obtain $ T^* \leq \lv E_1 ( 0 ) \rv / ( 2 \pi ) $. However, if the effect of the forcing term $ u $ is strong, a singular phenomena may occur in which a hole opens inside $ E_1 ( t ) $ within very short time from $ t = 0 $ as Figure 1. If $ E_1 ( t ) $ has a hole inside, since
	\[
		\frac{ d }{ d t } \lv E_1 ( t ) \rv = 0
	\]
	by the Gauss--Bonnet theorem until a hole shrinks to a point, the area of $ E_1 ( t ) $ hardly decreases at all for quite a long time and the extinction time can be made very large, even if $ \dv u = 0 $.
\end{remark}

\begin{figure}[h]
	\centering
	\includegraphics[width=0.8\linewidth]{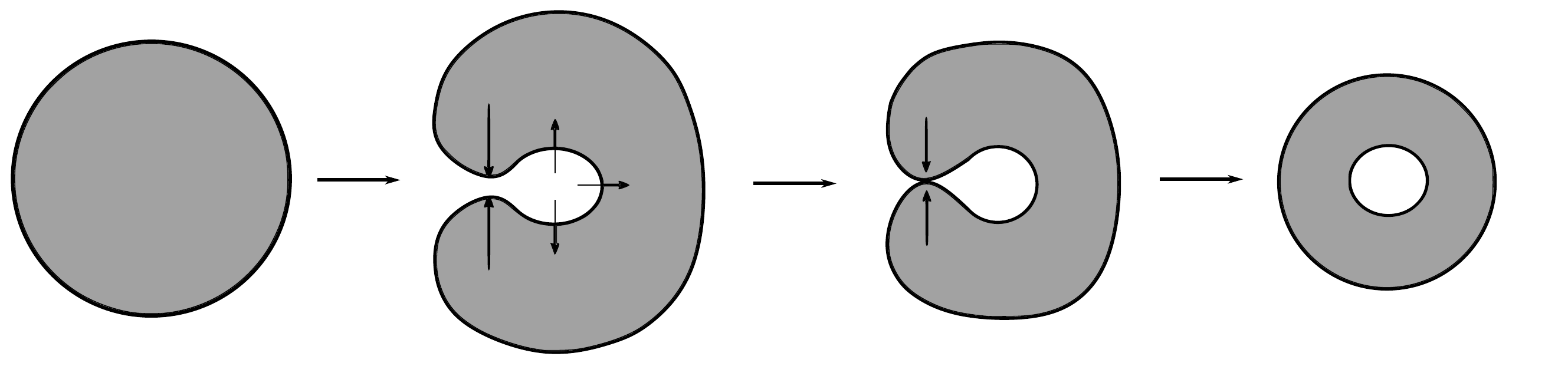}
	\caption{A large forcing term $ u $ may cause a hole in the domain.}
\end{figure}

\appendix

\section{Measure-Function Pairs}
Here, we recall the notion of measure-function pairs introduced by Hutchinson in \cite{hutchinson1986second}.
\begin{definition}
	Let $ E \subset \R^n $ be an open set and let $ \mu $ be a Radon measure on $ E $. Suppose $ f \in L^1 ( \mu ; \R^d ) $. Then we say that $ ( \mu , f ) $ is $ \R^d $-valued measure-function pair over $ E $.
\end{definition}
We define the notion of convergence for a sequence of $ \R^d $-valued measure-function pairs over $ E $.
\begin{definition}
	Let $ \{ ( \mu_i , f_i ) \}_{ i = 1 }^{ \infty } $ and $ ( \mu , f ) $ be $ \R^d $-valued measure-function pairs over $ E $. Suppose
	\[
	\mu_i \rightharpoonup \mu
	\] 
	as Radon measure on $ E $. Then we call $ ( \mu_i , f_i ) $ converges to $ ( \mu , f ) $ in the weak sense if
	\[
	\int_E f_i \cdot \phi\ d \mu_i \to \int_E f \cdot \phi\ d \mu
	\]
	for all $ \phi \in C^0_c ( E ; \R^d ) $.
\end{definition}
We present a less general version of \cite[Theorem 4.4.2]{hutchinson1986second} to the extent that it can be used in this paper.
\begin{theorem}
	\label{thm: measure function cptness}
	Suppose that $ \R^d $-valued measure-function pairs $ \{ ( \mu_i , f_i ) \}_{ i = 1 }^{ \infty } $ are satisfied
	\[
	\sup_i \int_E | f_i |^2\ d \mu_i < \infty.
	\]
	Then the following hold:
	\begin{description}
		\item[\textup{(1)}] There exist a subsequence $ \{ ( \mu_{ i_j } , f_{ i_j } ) \}_{ j = 1 }^{ \infty } $ and $ \R^d $-valued measure-function pair $ ( \mu , f ) $ such that $ ( \mu_{ i_j } , f_{ i_j } ) $ converges to $ ( \mu , f ) $ as measure-function pair.
		\item[\textup{(2)}] If $ ( \mu_{ i_j } , f_{ i_j } ) $ converges to $ ( \mu , f ) $ then
		\[
		\int_E | f |^2\ d \mu \leq \liminf_{ j \to \infty } \int_E | f_{ i_j } |^2\ d \mu_{ i_j } < \infty.
		\]
	\end{description}
\end{theorem}

\bibliography{myref.bib}
\bibliographystyle{alpha}

\end{document}